\documentclass[runningheads,twoside]{llncs}

\usepackage[ansinew]{inputenc}
\usepackage[T1]{fontenc}

\usepackage{multirow,bigdelim}   
\usepackage{graphicx}
\usepackage{amssymb}
\usepackage{textcomp}
\usepackage{amscd}
\usepackage{amsmath}
\usepackage{xcolor}

\usepackage{algorithm,algorithmic}
             
\begin{document}

\mainmatter              
\title{Sampling bipartite graphs with given vertex degrees and fixed edges and non-edges}

\titlerunning{Sampling bipartite graphs}
\author{Annabell Berger}
\institute{Martin Luther University Halle-Wittenberg, Halle (Saale), Germany\\
\email{annabell.berger@informatik.uni-halle.de}}
\authorrunning{A. Berger}
\maketitle

\begin{abstract}
We consider the problem of sampling a bipartite graph with given vertex degrees where a set $F$ of edges and non-edges which need to be contained is predefined. Our general result shows that the repeated swap of edges and non-edges in alternating cycles of at most size $2\ell-2$ ('$j$-swaps' with $j \leq 2 \ell-2$) in a current graph lead to an ergodic Metropolis Markov chain whenever $F$ does not contain a cycle of length $2 \ell$ with $\ell \geq 4.$ This leads to useful Markov chains whenever $\ell$ is not too large. If $F$ is a forest, $4$- and $6$-swaps are sufficient. Furthermore, we prove that $4$-swaps are sufficient when $F$ does not contain a matching of size $3.$ We extend the Curveball algorithm of Strona et al. \cite{Strona2014b} to our cases.
\end{abstract}

\section{Introduction}
\paragraph{The problem of sampling a bipartite graph for given vertex degrees} occurs in many fields and ways under several names. In the classical ecological book of Gotelli \cite{gotelli1996null} randomly sampled graphs with fixed degrees are proposed as null model for ecological networks (food webs, pollinator-plant networks, occurrence networks). They are used to answer questions as to how likely a link is between two species within a real network. Does a network's characteristics result from a hidden biological mechanism, or, does it occur by chance in such networks? One of the most common approaches is to compute an approximate sample with the use of a Metropolis Markov chain which is known as \emph{swap chain} or \emph{switching chain}. In the swap chain, one has to apply $t$ times the following procedure; either swap the ends of two non-adjacent edges if it is possible, or, stay in the current graph. It is not possible to swap the edge ends if there exists already one of the possible new generated edges. Exchanging two edge ends cannot change the vertex degrees of incident vertices, and hence generates a new graph with the same vertex degrees. We call such graphs \emph{realisations}. This approach is easy to implement, and the idea of swaps traces back to Petersen~\cite{Petersen1891}, in the 19th century. The swap chain itself was introduced by Diaconis and Gangolli~\cite{Diaconis1995}. The number $t$ of steps is called \emph{mixing time}, and describes 'how well all realisations are mixed' after $t$ steps. More precisely, the \emph{variation distance} between the probability distribution at step $t$, and the uniform distribution can not exceed a given $\epsilon$, and depends on the instance size (consider for more information the book of Levin et al.~\cite{LevinPeresWilmer2006}).\\
However, it seems to be quite hard to prove a \emph{rapid mixing time} for the swap chain, i.e. $t$ can be bounded by a polynomial which depends on the input size of the graph (edge number, vertex number) and $\epsilon$. Those upper bounds were proven for \emph{half-regular graphs} and graphs with bounded vertex degrees see~\cite{Kannan97,Miklos13,Greenhill:2015,Erdos16}. A bipartite graph $G=(U,V,E)$ is half-regular if the vertex degrees in $U$ or $V$ are identical. The general mixing time for the swap chain is still open, and the mixing times for the solved sub-classes are much too high to be applied. On the other hand, approximate sampling of realisations was proven to be rapid by Bez{\'a}kov{\'a} et al.~\cite{BezakovaBhatnagarVigoda07}. They extended the idea of Jerrum et al.~\cite{JerrumSinclairVigoda04} for sampling perfect matchings. However, all these polynomial bounds have only the theoretical value to show that there exists a polynomial approximate sampler. The exponents of these polynomials are much too high for any practical use. The reason seems to be a lack in our proving techniques. Moreover, many people believe that the swap chain is quite fast. Rechner et al~\cite{Rechner2016} found some experimental evidence for this assumption.

\paragraph{The Curveball chain} is a new Metropolis Markov chain to generate approximate samples at random. The idea of Strona et al~\cite{Strona2014b} was to apply several swaps for edges of two vertices at the same time. They had the nice idea to map, in a bipartite graph $G=(U,V,E),$ one vertex set $V$ to some kids, and the other vertex set $U$ to baseball cards. Edge set $E$ describes which cards are owned by which kid. A swap corresponds then to two kids, which exchange two cards, which they do not have in common. The natural question came up, why do they not change more cards at the same time. Intuitively, this must mix all cards much faster. More formally, the Curveball algorithm repeats the following steps $t$ times; (a) choose two adjacency lists $A_i$ and $A_j$ uniformly at random, (b) determine sets $A_{i-j}:=A_i \setminus A_j$ and $A_{j-i}$, and choose a subset $B_{i-j}$ of $A_{i-j} \cup A_{j-i}$ of size $|A_{i-j}|$ at random, denote the remaining subset by $B_{j-i},$ (c) replace $A_{i-j}$ in $A_i$ by $B_{i-j}$, and $A_{j-i}$ in $A_j$ by $B_{j-i}.$ Step (c) is called \emph{trade}, and can correspond to a swap if $A_{i-j}$ and $B_{i-j}$ only differ in one element. Carstens~\cite{CarstensPhysRevE} proved that this algorithm samples a realisation uniformly at random if $t \mapsto \infty.$ The theoretical status of the mixing times in the Curveball chain is open. Nevertheless, the proving techniques seem to work even worse for proving a rapid mixing time than in the swap chain. On the other hand the Curveball algorithm seems to mix much faster in experiments of Strona et al than the swap chain~\cite{Strona2014b}. Please observe that this new algorithm can apply exponentially many possible trades on one realisation whereas the swap chain can only use polynomially many swaps in one realisation. Intuitively, it is clear that more realisations can be achieved in less steps of the Curveball algorithm.\\
Independently, a similar approach was proposed by Verhelst~\cite{Verhelst2008}. In his paper, the Curveball algorithm is described briefly (in a complicated way) as an extension of his non-uniform sampling algorithm. However, no proof for its correctness is given. Instead, he proposes a Metropolis-Hasting version which basically avoids to consider adjacency list pairs where no trade is possible. Intuitively, this approach leads to a smaller mixing time than the Curveball algorithm because one has to stay in a current realisation less frequently. On the other hand the asymptotic complexity of one step in Verhelst`s algorithm is much higher than in the Curveball algorithm. It is questionable (or unclear) if the smaller mixing time can balance out the calculation effort for one step. In our paper we focus on the Curveball algorithm.\\
We do not want to focus in this paper on the problem to determine good upper bounds for the mixing time. Instead, we consider an extension of the problem. Consider for example a plant-pollinator network describing which pollinator pollinates which plant. Then it is clear that some edges can never occur because a bee is too large for a flower, or, their activity time periods are different. Hence, a more realistic null model would fix some non-edges and edges in a given realisation, and only search for random samples in this subclass of graphs. 

\paragraph{Sampling of $F$-fixed bipartite graphs with fixed degree sequence} We define the problem of sampling a bipartite graph with fixed degree sequence where a set $F_E$ of edges and a set $F_N$ of non-edges is already fixed. 

\begin{definition}[$F$-fixed realisation]\label{def:F_fixed_realisation}
We call a labeled bipartite graph $G=(V,U,E)$ without parallel edges and loops an \emph{$F$-fixed realisation of integer sequence $S=(a_1,\dots,a_n),(b_1,\dots,b_{n'})$} with $a_i,b_i>0$ for set $F:= F_E \cup F_N$ if and only if
\begin{enumerate}
\item $d_G(v_i)=a_i$ for all $v_i \in V$, $d_G(u_i)=b_i$ for all $u_i \in U$, and 
\item $F_E \subset E$ and $F_N \cap E = \emptyset$.
\end{enumerate}
\end{definition}

A special version of this problem is that all elements in $F$ are non-edges. This problem is the classical \emph{bipartite f-Factor problem}. Finding a suitable $f$-Factor can be done in polynomial time. For an overview consider the book of Schrijver~\cite{schrijver2003}. However, sampling a suitable $f$-factor has not received much attention. Let us denote the symmetrical difference of the edge sets $E(G)$ and $E(G')$ of two $F$-fixed realisations $G$ and $G'$ by $G \triangle G':=E(G) \triangle E(G').$ Since $G$ and $G'$ possess the same vertex degrees, $G \triangle G'$ decomposes in \emph{alternating cycles of $G$ and $G'$}, i.e. each cycle $C_i$ consists alternately of edges from $G$ and $G'.$ A cycle $C_i$ is not necessarily vertex disjoint since both graphs can be different regarding several adjacent edges of a vertex. Generally we define a cycle with edges alternately in $G$ and not in $G$ by \emph{alternating cycle of $G$}. That is in these cases we do not say specifically where the non-edges are from. Deleting all edges in $G$ which belong to its alternating cycle and introducing all non-edges in this cycle as edges, leads to a new realisation $G^{*}.$ We call the processing of one alternating cycle of $G$ to achieve another realisation $G^{*}$ \emph{cycle swap}.\\
Please observe that cycle swaps can never contain edges or non-edges of $F$ since they are fixed in each realisation and so not in the symmetrical difference. It is not important if an element of $F$ is an edge or an non-edge. That is we can treat each $F$-fixed problem like an $f$-factor problem and vice versa. Recall that a cycle swap, which is not vertex disjoint, decomposes in vertex disjoint, alternating cycles of $G$ since every vertex must possess the same number of adjacent edges in $E(G)$ and $E(G').$ We call cycle swaps of length $k,$ which are vertex disjoint, \emph{$k$-swap}.  Hence, each realisation $G'$ can be achieved by repeated $k$-swaps of lengths $4 \leq k \leq 2 \min \{|U|,|V|\}.$ For an example consider Figure~\ref{fig:swap_of_alternating_cycles}.

\begin{figure}
\label{fig:swap_of_alternating_cycles}
 \centering
 \includegraphics[scale=0.3]{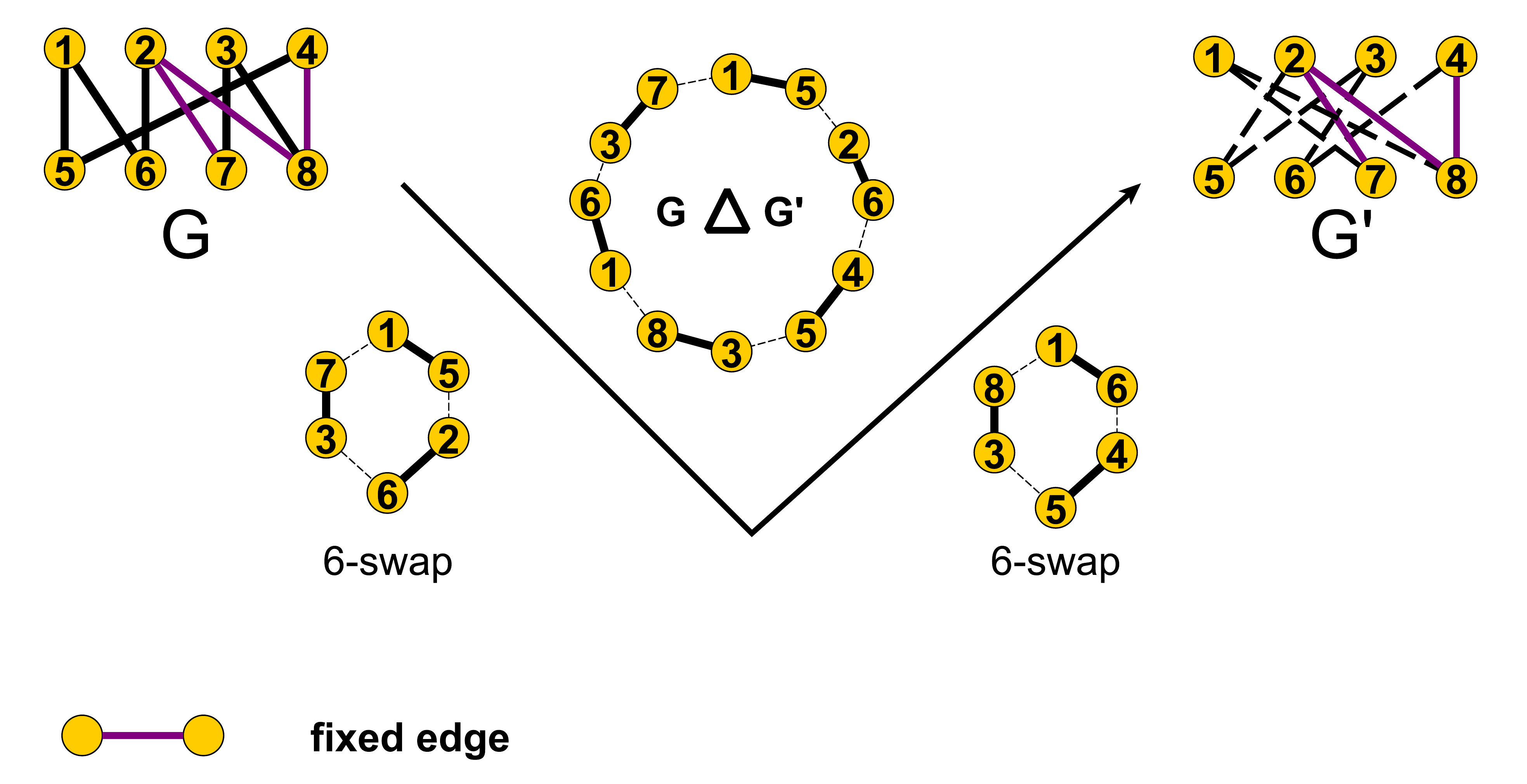}
\caption{Decomposing of cycle swap $G \triangle G'$ in $k$-swaps with $k=6.$ Purple edges are in $F.$ }
\end{figure}

 This result was first given by Erd\H{o}s et al.~\cite{Erdoes15} for $f$-Factors. Unfortunately, this insight is not applicable in a sampling algorithm since finding all alternating cycles in a given $f$-Factor $G$ contains the NP-complete decision problem to decide if a graph possesses a Hamiltonian cycle. However, Erd\H{o}s et al.~\cite{Erdoes15} also give for a very special $f$-Factor an ergodic sampling chain where the forbidden non-edges represent a perfect matching, and a star on one vertex. Their Markov chain is ergodic when they use $4$-swaps and $6$-swaps. We improve these results in several ways.

\paragraph{Our contribution} We not only consider $f$-Factor problems but problems where the fixed set $F$ consists of edges and non-edges. Furthermore, we refine the general result of Erd\H{o}s et al. and show that it is sufficient to use $k$-swaps where $4 \leq k \leq 2 \ell -2$ if set $F$ does not contain a cycle of length $2 \ell$ where $\ell \geq 4.$ For the case where $\ell=4,$ we find that we can use $6$-swaps and $4$-swaps to yield an ergodic Markov chain whenever set $F$ does not contain a cycle of length $8.$ This leads, for example, to an applicable ergodic chain for bipartite graphs if, $F$ is a forest. Moreover, we extend the Curveball algorithm to a version which applies several $6$-swaps at the same time. For the special case where $F$ does not contain a matching of size three we prove that the use of $4$-swaps lead to an ergodic Markov chain. We show that the classical swap chain or the Curveball algorithm ignoring the edges of $F$ is an ergodic Markov chain. Lastly, we generalize all these results using a preprocessing step which calculates the set $F'$ of edges and non-edges for a fixed degree sequence which occur in \textbf{each} realisation. Given an  $F$-fixed problem we can delete in $F$ all elements which occur in $F'$ since each $F$-fixed realisation is also an $F\setminus F'$-fixed realisation and vice versa. This approach improves the chance that the the upper bound $2 \ell-2$ for suitable $k$-swaps is small enough to find them efficiently. The reason is that there is a certain chance that $F \setminus F'$ does not contain a cycle of length $2 \ell$ which occurs in $F.$ 

\paragraph{Structure} The remainder of this article is structured as follows. In the next section we introduce a Curveball algorithm and swap chain for sampling $F$-fixed realisations for a fixed set $F$ which does not contain a matching of size three. In the second part of this paper, we propose a Curveball algorithm with cycle trades which allows us to apply several $6$-swaps at the same time. We prove that this algorithm can be used for all problems where $F$ does not contain a cycle of length $8.$ Furthermore, we generalize this result for all $F$-fixed problems. In the summery we give instruction on how to come from a given $F$-fixed problem to a suitable sampling algorithm.

\section{$F$-fixed problems with swaps}\label{sec:swaps}

The following result allows us to construct an ergodic Markov chain which is based on $4$-swaps which exclude edges from $F$. Specifically, we show that the symmetrical difference $G \triangle G'$ decomposes in a series of $4$-swaps whenever the fixed set $F$ does not contain matchings of size $3.$

\begin{theorem}\label{theorem:connectedness_F_Fixed_realizations_swaps}
Let $G=(V,U,E)$ and $G'=(V,U,E')$ be two different $F$-fixed realizations of a sequence $S$ where the fixed set $F:=F_E \cup F_N$  does not contain a matching of size $3$. Then there exist $F$-fixed realisations $G_1,\dots,G_k$ with $G_1:=G$, $G_k:=G'$ and $|G_i \triangle G_{i+1}|=4$ where $G_i \triangle G_{i+1}$ corresponds to a $4$-swap of $G_i$ and $k\leq \frac{1}{2}|G \triangle G'|$.
\end{theorem}

\begin{proof}
\emph{Each alternating walk of length three in $G \triangle G'$ must be vertex-disjoint} since a) two vertices on such a walk with distance one cannot be identical since loops are forbidden, b) two vertices on such a walk with distance two cannot be identical otherwise it contradicts the definition of a symmetrical difference, c) two vertices on this cycle with distance three cannot be identical because in bipartite graphs one of those vertices belongs to $U$ and the other one to $V.$ We prove the result with induction by size $\kappa$ where $|G \triangle G'|=2\kappa.$\\
For $\kappa=2$ the result is true setting $G_1:=G$, $G_k:=G'$ and $k=2.$ For $\kappa=3$ we get an alternating cycle of length $6.$ This cycle is vertex-disjoint because each walk of length three on this cycle is vertex-disjoint as proven above and two vertices on such a cycle have at most a distance of three. We denote this cycle by $C=(v_1,v_2,v_3,v_4,v_5,v_6,v_1)$ with $\{v_1,v_2\},\{v_3,v_4\},\{v_5,v_6\}\in E(G)$ and $\{v_2,v_3\},\{v_4,v_6\},\{v_6,v_1\} \in E(G').$ There must be a vertex pair $\{v_i,v_j\} \notin F$ which connects an alternating $3$-path in $C.$  Otherwise, we find the matching $\{v_1,v_4\},\{v_2,v_5\},\{v_3,v_6\}$ of size $3$ in $F.$ Without lost of generality we assume that this pair is $\{v_1,v_4\}.$ Then we find either $\{v_1,v_4\} \in E(G) \cap E(G')$ or $\{v_1,v_4\} \notin E(G) \cup E(G').$ The first case leads to $4$-swap $C':=(v_1,v_4,v_5,v_6,v_1).$ Its application leads to $F$-fixed realization $G^*$ where $G^* \triangle G'$ is the $4$-swap $(v_1,v_2,v_3,v_4,v_1).$ Setting $G_1:=G$, $G_2:=G^*$ and $G_3:=G'$ leads to our assumption. The second case with $\{v_1,v_4\} \notin E(G) \cup E(G')$ can be treated analogously by first swapping the alternating cycle $C':=(v_1,v_2,v_3,v_4,v_1)$ and then $C'':=(v_1,v_4,v_5,v_6,v_1)$. Furthermore, $k=3=\frac{1}{2}|G \triangle G'|.$\\
Let us now assume the claim is correct for all $G$ and $G'$ with $|G \triangle G'|=2 \kappa$ where $\kappa \leq n.$ We prove the correctness for $|G \triangle G'|=2 n+2$. $G \triangle G'$ is an Eulerian graph decomposing in alternating cycles. We assume that $G \triangle G'$ is connected. Otherwise we can apply the induction hypothesis on each of the cycles separately and are done. We denote this alternating cycle by $C=(v_1,\dots,v_{2n+2},v_1)$ with edge $\{v_1,v_2\} \in E(G).$ We consider the alternating sub-walk  $P=(v_1,\dots,v_6)$ of $C.$ \emph{If $P$ is not vertex-disjoint} it can happen that we get in $P$ only $5$ different vertices; either $v_1=v_5$ or $v_2=v_6.$ Less than this is not possible since each path of length three must be vertex-disjoint as proven above. Without lost of generality we consider $v_1=v_5.$ Then $C$ decomposes in alternating cycles $C_1=(v_1,\dots,v_4,v_1)$ and $C_2=(v_1,v_6,\dots,v_{2n+2},v_1)$ of lengths $\ell_1=4$ and $\ell_2=2n-2.$ We apply $4$-swap $C_1$ on $G$ and yield $G^*.$ $G^* \triangle G'$ corresponds to $C_2.$ We apply induction hypothesis on $|G^* \triangle G'|$ and get sequence of $F$-fixed realisations $G_1,\dots,G_k$ with $G_1:=G$ and $G_k:=G^*$ with $k \leq n-1.$ Hence, $G,G_1,\dots,G_k$ is the required sequence of length $n$ for $G$ and $G'.$ \emph{If $P$ is a vertex-disjoint walk} there must be a vertex pair $\{v_i,v_j\}\notin F$ such that there is a path of length three from $v_i$ to $v_j$ on $P.$ Otherwise we find $(v_1,v_4),(v_2,v_5),(v_3,v_6) \in F$ which is a matching of size three in $F$ contradicting our assumption. 
Without lost of generality we assume that $e=\{v_1,v_4\} \notin F.$ Hence, $e$  connects the sub-path $P_1=(v_1,v_2,v_3,v_4)$ of $C$. (The case where $e=\{v_2,v_5\}$, or $e=\{v_3,v_5\}$ can be treated analogously.) We denote the remaining sub-walk of $C$ by $P_2=(v_4,\dots,v_{2 \ell +2},v_1).$ Recall that $e \notin G \triangle G'$ otherwise $C$ is not vertex disjoint. We get the following two cases.\\
(1) $\{v_1,v_4\} \in E(G) \cap E(G')$: We apply a cycle swap of length $2 n-2$ on alternating cycle $C'=(v_1,P_2)$, and get $F$-fixed realisation $G^*$ with $|G^* \triangle G'|=4$ and $|G^* \triangle G|=2n.$\\
(2) $\{v_1,v_4\} \notin E(G)\cup E(G')$: We apply a cycle swap of length $4$ on alternating cycle $C'=(P_1,v_1)$, and get $F$-fixed realisation $G^*$ with $|G^* \triangle G'|=2n$ and $|G^* \triangle G|=4.$\\
In the first case we apply on $C'$ the induction hypothesis and get a sequence $G_1,\dots,G_k$ of fixed $F$-realisations with $G_1:=G$ and $G_{k}=G^*$ with $k=n.$ Setting $G_{k+1}:=G'$ leads to the expected result. In the second case we apply on $C'':=G^* \triangle G'$ the induction hypothesis and get a sequence $G_2,\dots,G_{k}$ of fixed $F$-realisations with $G_2:=G^*$ and $G_{k}=G'$ with $k=n+1.$  Setting $G_1:=G$ leads to the expected result.
 \qed\end{proof}

We extend the Curveball algorithm of Strona et al.~\cite{Strona2014b} to an \emph{F-ignoring Curveball algorithm}, i.e. this algorithm ignores in a trade set $F.$ Let us consider the two adjacency lists $A_1:=\{1,2,3,4,5,6\}$ and $A_2:=\{4,5,7\}$ and $F$ consists of fixed edge $\{v_1,u_3\}$ and fixed non-edge $\{v_2,u_6\}.$ Then $A_1$ and $A_2$ are only allowed to exchange numbers in set $(A_1 \triangle A_2) \setminus \{3,6\}=\{1,2,7\}.$ Especially, we can replace in $A_1$ either $1$ by $7$, or $2$ by $7.$ Even if $6 \notin A_2$ we cannot swap it with a number in $A_2$ because the non-edge $\{v_2,u_6\}$ is fixed. More formally, we determine regarding $F$ an adjacency lists $F_i$ for each vertex $v_i \in V.$ It is important to notice that all $F_i$s contain neighbours as well as \textbf{non-neighbours} of vertex $v_i$ in $F.$ Recall that the original Curveball algorithm applies $t$ times a trade between two randomly chosen adjacency lists of $G.$ We refine such a trade to \emph{$F$-ignoring trades} in the following way. For two adjacency lists $A_i$ and $A_j,$ we define the sets $A_{i-j}$ and $A_{j-i}$ by $A_{i-j}:=(A_i\setminus (F_i \cup A_j \cup F_j).$ In an $F$-ignoring trade we choose a subset $B_{ij}$ of size $|A_{i-j}|$ in $A_{i-j} \cup A_{j-i}$ uniformly at random, and put the remaining elements in set $B_{ji}.$ Then we delete in $A_i$ set $A_{i-j}$ and add set $B_{ij}.$ We do the analogous procedure in $A_j.$ The $F$-ignoring Curveball algorithm starts with an $F$-fixed realisation $G$ and applies $t$ times the following procedure; (a) choose in $G$ uniformly at random two adjacency lists $A_i$ and $A_j$, (b) choose an $F$-ignoring trade for $A_i$ and $A_j$ uniformly at random, and go to $F$-fixed realisation $G'.$ It can occur there is no suitable trade for two lists. In this case stay in the current realisation $G.$  As a special case we define an  \emph{F-ignoring swap algorithm}. That is we choose $B_{ij}$ in the above algorithm only of size $1.$ Back to our example we find with $F_1=\{3\}$ and $F_2=\{6\}$ that $A_{1-2}=\{1,2\}$ and $A_{2-1}=\{7\}.$ Hence, we can choose in an $F$-ignoring Curveball algorithm three different trades, i.e., the swap of elements $1$ and $7$, of $2$ and $7$, or, we stay in the current realisation in choosing $B_{ij}=\{1,2\}.$ Please observe that each trade conserves the number of elements in $A_1$ and $A_2$ which can be exchanged with each other. This is a fundamental observation for the convergence of the Markov chain to the uniform distribution as we show in the next proof.

\begin{corollary}\label{Cor:Curveball_F_fixed}
The $F$-ignoring swap algorithm and the $F$-ignoring Curveball algorithm sample $F$-fixed realisations uniformly at random for $t \mapsto \infty$ if $F$ does not contain a matching of size three.
\end{corollary}

\begin{proof}
The fundamental Theorem of Markov chains, see for example~\cite{LevinPeresWilmer2006}, states that this algorithm samples an $F$-fixed realisation uniformly at random for $t \mapsto \infty$, if the Markov chain is irreducible, reversible and aperiodic. Since there is always a swap sequence between two realisations the chain is irreducible because we can define a sequence of trades (which are $4$-swaps) between each pair of realisations with Theorem~\ref{theorem:connectedness_F_Fixed_realizations_swaps}. A trade between two realisations $G$ and $G'$ can always be done in both directions. Hence, the state graph is symmetrical. We need to prove that we find transition probabilities $P(A,B)=P(B,A)$ for all matrices $A$ and $B$ which differ by one $F$-ignoring trade. In this case, the Markov chain is reversible, i.e. $\pi(A)P(A,B)=P(B,A)\pi(B)$ and the stationary distribution $\pi$ must be uniform. Let us consider sets $A_{i-j}$ and $B_{i-j}$ for two matrices $A$ and $B$ which differ by one $F$-ignoring trade for adjacency lists $i$ and $j$. The transition probability from $A$ to $B$ is $P(A,B)=\frac{2}{n \cdot (n-1)}{|A_{i-j} \cup A_{j-i}|\choose |A_{i-j}|}^{-1}.$ We show that $|A_{i-j}|=|B_{i-j}|$ and $|A_{j-i}|=|B_{j-i}|.$ Since $B_{j-i} \cap B_{i-j}=\emptyset$ it follows $P(A,B)=P(B,A).$ First, observe that for matrices $A$ and $B,$ which differ by one trade for adjacency lists of $v_i$ and $v_j$, we find that they have the same set of $1$'s which are exchangeable. This means $A_i \triangle A_j=B_i \triangle B_j.$ Especially the number of exchangeable $1$'s in $A_i$ must be the same as in $B_i$ because an $F$-ignoring trade swaps a number of elements in $A_i$ in taking the same number of elements from $A_j.$ Since the elements in $F_i$ and $F_j$ are ignored we find $|A_{i-j}|=|A_i\setminus (A_j \cup F_i \cup F_j)|=|B_i\setminus (B_j \cup F_i \cup F_j)|=|B_{i-j}|$ and also $|A_{j-i}|=|A_j\setminus (A_i \cup F_j \cup F_i)|=|B_j\setminus (B_i \cup F_j \cup F_i)|=|B_{j-i}|.$ The chain is aperiodic since we always find a $4$-swap in matrix $A$ whenever there are at least two $F$-fixed realisations with Theorem~\ref{theorem:connectedness_F_Fixed_realizations_swaps}. This swap corresponds to a trade between two list $A_i$ and $A_j.$ But then we also have the trade where $B_{i,j}=A_{i-j}$, and stay in the current realisation.
\qed\end{proof}

Sometimes it can happen that a certain set $F'$ of edges and non-edges for a given degree sequence is contained in \textbf{each} realisation. Searching for an $F'$-fixed realisation could be implemented in searching for a realisation without any conditions since edges and non-edges of $F'$ will be contained in each realisation anyway. This approach leads to a very useful extension of Theorem~\ref{theorem:connectedness_F_Fixed_realizations_swaps}. We can extend the result to all $F \cup F^{*}$-fixed sets where $F$ does not contain a matching of size three, and $F^{*} \subset F'$. We denote the set $F'$ by \emph{the static edge and non-edge set} of sequence $S$ which is a unique set. Given an $H$-fixed realisation problem one can first investigate if there exist a partition of $H$ in $F$ and $F^{*}$ such that $F$ does not contain a matching of size three and $F^{*} \subset F'$. For that we have to determine set $F'.$ This can be done by a repeated application of the Gale-Ryser theorem, see~\cite{gale1957,ryser1957}, which gives sufficient and necessary conditions that a sequence has a bipartite realisation.\\
 We propose the following approach. For each vertex pair $\{v_i,u_j\} \in V \times U$ we reduce sequence $S=(a_1,\dots,a_n),(b_1,\dots,b_{n'})$ to $S'$ in setting $a_i':=a_i-1$ and $b_j':=b_j-1.$ Then we apply the Gale-Ryser theorem on $S'$. If $S'$ is not realisable we can conclude that $\{v_i,u_j\}$ belongs to the non-edges in set $F'$. We apply the same approach for each vertex pair to the \emph{opposite sequence} $\overline{S}:=(n'-a_1,\dots,n'-a_n),(n-b_1,\dots,n-b_{n'}).$ If it turns out that $\{v_i,u_j\}$ cannot be realized for sequence $\overline{S},$ then $\{v_i,u_j\}$ is contained in each realisation of $S$, and so belongs to the edge set of $F'.$ It is not necessary to consider all vertex pairs $\{v_i,u_j\}.$ At the beginning it is useful to consider an arbitrary realisation $G.$ A vertex which is connected to all possible vertices can be deleted and the sequence be reduced since all these edges will occur in each further realisation. However, it is more useful to determine all possible $4$-swaps in $G.$ Edges and non-edges of those swaps can change their role in realisations, and so not belong to the static edge and non-edge set $F'$ of $S.$ These edges do not need to be considered in the Gale-Ryser test. The general result for $H$-fixed realisations can be given in the following Theorem~\ref{theorem:connectedness_F_F'_Fixed_realizations_swaps}.

\begin{theorem}\label{theorem:connectedness_F_F'_Fixed_realizations_swaps}
Let $G=(V,U,E)$ and $G'=(V,U,E')$ be two different $H$-fixed realizations of a sequence $S$ where $H$ can be partitioned in sets $F$ and $F^{*}$ such that 
\begin{enumerate}
\item $F$ does not contain a matching of size $3$, 
\item $F^{*} \subset F'$ where $F'$ is the static edge and non-edge set of $S$, and 
\item $F \cap F^{*} = \emptyset$.
\end{enumerate}
Then the $H$-ignoring swap and Curveball algorithms sample an $H$-fixed realisation uniformly at random for $t \mapsto \infty.$
\end{theorem}

\begin{proof}
Each $F$-fixed realisation is equivalent to an $H$-fixed realisation. Hence, Corollary~\ref{Cor:Curveball_F_fixed} ensures that an $F$-ignoring Curveball algorithm samples an $H$-fixed  realisation at random. Since suitable trades cannot contain edges from $F^*,$ an $H$-fixed swap chain samples all $H$-fixed realisations.
\qed\end{proof}

\section{General $F$-fixed Markov chains, and chains with swap cycles of length six}

We start with a result which will be needed in our main theorem.

\begin{proposition}\label{prop:cycles_in_F}
Given an even vertex disjoint cycle $C=(v_0,\dots,v_{2n-1},v_0)$ of length $2n \geq 8.$ Let $A$ be the edge set of all vertex pairs $\{v_i,v_j\}$ on $C$ such that $v_i$ and $v_j$ have an odd distance on $C$ of at least length three. Then $A$ contains each cycle of length $8 \leq 2 \ell \leq 2n.$
\end{proposition}

\begin{proof}
We proof the claim with induction by the length $2n$ of cycle $C$. We start with the smallest possible length $8.$ Then $C'=(v_0,v_3,v_6,v_1,v_4,v_7,v_2,v_5,v_0)$ is a cycle in $A$ of length $8.$ Let us assume cycle $C=(v_0,\dots,v_{2n+1},v_0)$ has length $2n+2.$ We construct the cycle $C'=(v_0,\dots,v_{2n-1},v_0)$ with $\{v_{2n-1},v_0\}\in A.$ Set $A':=A \setminus \{\{v_{2n-1},v_0\},\{v_{2n},v_1\},\{v_{2n+1},v_2\},\{v_{2n-3,2n}\},\{v_{2n-2},v_{2n+1}\}\}$ is the edge set of all vertex pairs $\{v_i,v_j\}$ on $C'$ such that $v_i$ and $v_j$ connect an odd path on $C'$ of at least length three. With the induction hypothesis it follows that $A'$ contains each cycle of length $8 \leq 2 \ell \leq 2n.$ Since $A' \subset A$, $A$ contains all these cycles too. It remains to prove that $A$ contains a cycle of length $2n+2.$ We define the function $f:\{0,\dots,2 n+1\} \mapsto \{0,\dots,2 n+1\}$ with $f(i)= (t \cdot i) \pmod{2n+2}$ such that $2n-1 \geq t \geq 3$, $t$ is odd, and the pair $t$, $2 n+2$ is relatively prime.\\ 
Such a pair exists what we prove with the Eulerian $\phi$-function counting the number $\phi(2n+2)$ of relatively prime numbers for $2n+2$ in $M:=\{1,\dots,2n+1\}.$ If $\phi(2n+2) \geq 3$ for all $n \geq 4$ we are done. The reason is that number $1$ is always relatively prime to $2n+2.$ Hence, there are two other elements in $M$ which are relatively prime to $2n+2.$ Since $2n+2$ is even these elements need to be odd. This leads to the existence of a $t \in M$ with $2n-1 \geq t \geq 3$ which is relatively prime to $2n+2.$ Assume now that $\phi(2n+2) \leq 2.$ With the definition of $\phi$, we have $\phi(2n+2)=\displaystyle\prod_{p} p^{k_p-1}(p-1)$ where $2n+2=\displaystyle\prod_{p} p^{k_p}$ is the prime factorisation of $2n+2.$ If $2n+2=2^k$ we get $\phi(2n+2)=2^{k-1}\leq 2.$ It follows $k \leq 2$, and so $2n+2 \leq 4$ in contradiction to $n \geq 4.$ Hence, $2n+2$ must be of the form $2n+2=2 \cdot\displaystyle\prod_{p\neq 2} p^{k_p}$ leading to $\phi(2n+2)=\phi(n+1) \leq 2.$ The condition $\phi(n+1)=\displaystyle\prod_{p\neq 2} p^{k_p-1}(p-1)\leq 2$ can only be fulfilled if $p-1 \leq 2$ and $k_p=1.$This is only possible for $n+1=3$ in contradiction to $n \geq 4.$\\
Function $f$ must be bijective. Otherwise we find $f(i)=f(j)$ with $i<j.$ That is $j$ must be of the form $j=i+(2n+2)k$ where $k >0 $ is an integer. This leads to $j\geq 2n+2$ for $i \in \{0,\dots,2n+1\}.$ Hence, $j$ does not occur in the defined interval for $f$ in contradiction to the definition of $f.$ We construct cycle $C^*=(v_{f(0)},v_{f(1)},v_{f(2)},\dots,v_{f(2n+1)},v_{f(0)})$ of length $2 n+2.$ $C^*$ has all edges in $A$, because it alternates between odd and even labeled vertices, i.e. for even $i$ the value $f(i)$ is even, and vice versa for odd $i$. Furthermore, two adjacent vertices in $C^*$ have an odd distance of $2n-1 \geq t \geq 3.$ Since $f$ is bijective, cycle $C^*$ must be vertex-disjoint. 
\qed\end{proof}

Figure~\ref{fig:circle} shows a vertex disjoint cycle $C$ with 12 vertices. Edge set $A$ is given by all vertex pairs $\{v_i,v_j\}$ on $C$ such that $v_i$ and $v_j$ connect an odd path on $C$ of at least length three. We construct cycles in $A$ of lengths $8$, $10$ and $12$ in setting $t=3$ for the first two cases and $t=5$ for the last case.

\begin{figure}\label{fig:circle}
 \centering
 \includegraphics[scale=0.3]{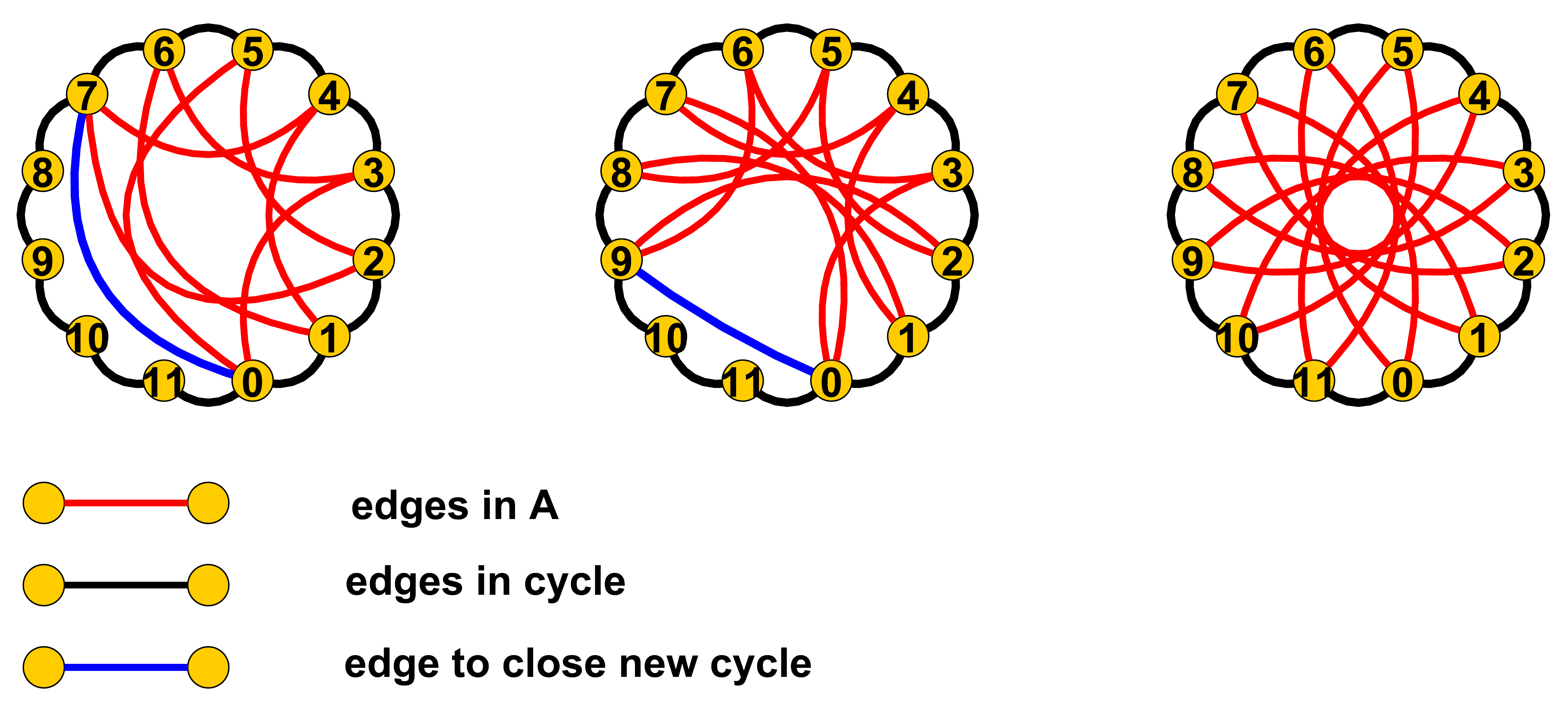}
\caption{Construction of $A$-cycles with length $8$, $10$ and $12.$ }
\end{figure}

\begin{theorem}\label{Theorem:connectedness_F_Fixed_realizations_cycles_length_2k}
Let $G=(V,U,E)$ and $G'=(V,U,E')$ be two different $F$-fixed realisations of a sequence $S$ such that $F$ does not contain a cycle of length $2 \ell$ where $\ell \geq 4$. Then there exist realizations $G_1,\dots,G_k$ with $G_1:=G$, $G_k:=G'$ such that (i) $|G_i \triangle G_{i+1}|\leq 2\ell-2$ where $G_i \triangle G_{i+1}$ corresponds to a $j$-swap with $j \leq 2 \ell-2$ , and (ii) $k\leq \frac{1}{2}|G \triangle G'|.$
\end{theorem}

\begin{proof}
We prove the statement with induction by the size of the symmetrical difference $\kappa:= \frac{1}{2}|G \triangle G'|.$ 
\paragraph{induction basis.} For each fixed $\ell$ the cases with $2 \leq \kappa \leq \ell-1$ are simple. $G \triangle G'$ decomposes in alternating vertex-disjoint cycles between $G$ and $G'$ of at most size $2\ell-2.$ Hence, condition (i) is fulfilled by the sequence $G_1,\dots,G_k$ with $G_1:=G,$ $G_k:=G'$ such that each $G_i \triangle G_{i+1}$ corresponds to a $j$-swap with $j \leq 2 \ell-2.$ In the case where $G \triangle G'$ decomposes in alternating $4$ cycles, $k$ is maximum with $k=\frac{1}{4}|G \triangle G'|$ fulfilling (ii). To continue the induction basis we set $\kappa=\ell,$ i.e. $|G \triangle G'|=2 \ell.$ If $G \triangle G'$ is not connected, each of the components is an alternating cycle of at most length $2 \ell -4$, decomposing in $j$-swaps with $j \leq 2\ell-4.$ Moreover, index $k$ cannot be larger than $k=\frac{1}{4}|G \triangle G'|.$ We now assume that $G \triangle G'$ is one alternating $2 \ell$-cycle $C:=(v_0,\dots,v_{2 \ell-1},v_0)$ between $G$ and $G'$ which can be (a) vertex-disjoint, or, (b) contains at least one vertex, say $v_1,$ twice. W.l.o.g. we assume that $C$ starts with edge $\{v_0,v_1\}\in E(G).$ 
\emph{We consider case (a)}, see Figure~\ref{fig:induction_beginning}, and assume that all vertex pairs $\{v_i,v_j\}$ on $C$ which have an odd distance of at least $3,$ are arcs in the set $F$. With Proposition~\ref{prop:cycles_in_F} we find a cycle of length $2 \ell$ in $F$ in contradiction to our condition. Hence, there must be a vertex pair $\{v_i,v_j\}$ with an odd distance of at least $3$ on $C$  which does not belong to $F.$ We assume $i<j$, and denote the alternating sub-path on $C$ from $v_i$ to $v_j$ by $P_1=(v_i,v_{i+1},\dots v_j)$ and its length by $\ell_1.$ The remaining odd alternating sub-path on $C$ from $v_j$ to $v_i=v_{j+(2\ell-\ell_1 \pmod{2n})}$ by $P_2=(v_{j\pmod{2n}},v_{j+1\pmod{2n}},\dots,v_{j+(2\ell-\ell_1 \pmod{2n})}).$ Since $\{v_i,v_j\} \notin G \triangle G'$, we have either $\{v_i,v_j\} \in G \cap G'$ or $\{v_i,v_j\} \notin G \cup G'.$ We find that either $C':=(v_i,P_2)$ or $C'':=(P_1,v_i)$ is an $(\ell_2+1)$-swap or $(\ell_1+1)$-swap, respectively. Moreover, we have $\ell_1+1,\ell_2+1 \leq  2 \ell-2$ and get $F$-fixed realisation $G^{*}.$ Furthermore, we find that $G^{*} \triangle G'$ corresponds to a $j$-swap $C^*=(P_1,v_i)$ or $C^{**}=(v_i,P_2)$, respectively. Additionally, we have  $j\leq 2 \ell-2.$ We find a sequence $G_1,G_2,G_3$ of $F$-fixed realisations such that $G_1:=G$, $G_2:=G^{*}$, $G_3:=G'$, and (i) is fulfilled. Condition (ii) is fulfilled because $k=3\leq \ell.$

\begin{figure}
 \centering
 \includegraphics[scale=0.25]{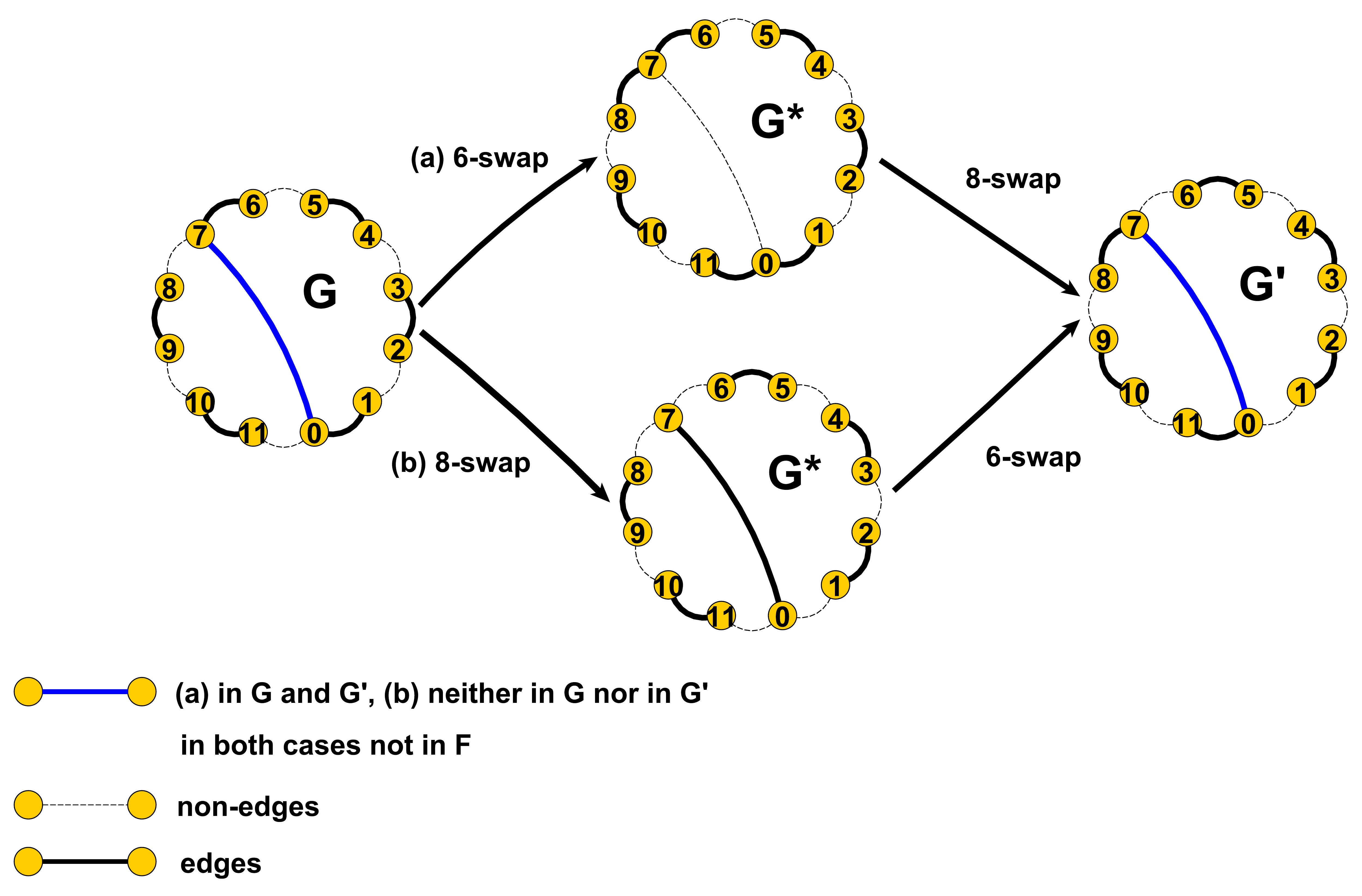}
\caption{induction basis $\kappa= \ell$. Case (a) with $2\ell=12$ and $|G \triangle G'|=12.$}
\label{fig:induction_beginning}
\end{figure}

\emph{We consider case (b)} see Figure~\ref{fig:induction_beginning_2} where the alternating cycle $C$ contains at least one vertex twice. Then $C$ decomposes in at most $\frac{1}{2} \ell$ $j$-swaps with $4 \leq j \leq 2 \ell-4.$ Hence, (i) in our Theorem is fulfilled with sequence $G_1,G_2,\dots,G_{k}$ where $G_1:=G$, $G_{k}:=G'$ and $|G_i \triangle G_{i+1}|\leq 2 \ell-4$ corresponds to one of these $j$-swaps. Condition (ii) is fulfilled because $k \leq \frac{1}{2}\ell +1= \frac{1}{4}|G \triangle G'|+1\leq \frac{1}{2}|G \triangle G'|$ for $|G \triangle G'| \geq 8$ as in our assumption.

\begin{figure}
 \centering
 \includegraphics[scale=0.2]{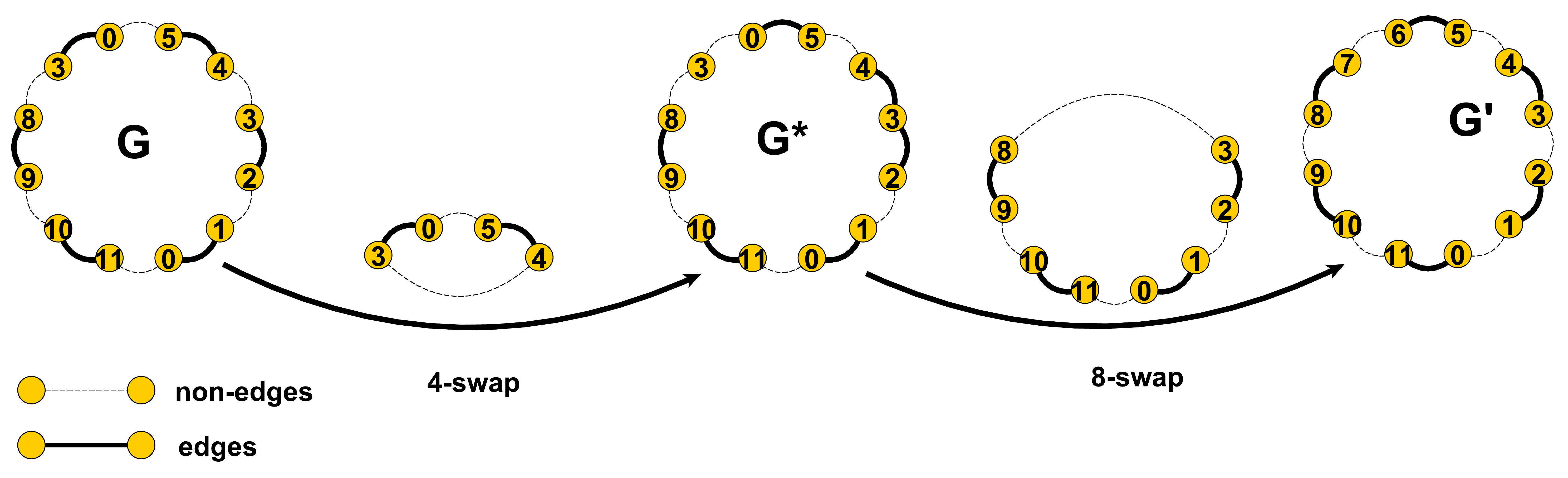}
\caption{induction basis $\kappa= \ell$. Case (b) with $2\ell=12$ and $|G \triangle G'|=12.$}
\label{fig:induction_beginning_2}
\end{figure}

\paragraph{Induction Step.} We only need to consider cases where $\kappa>\ell \geq 4.$ We assume that the claim is right for all $F$-fixed realisations $G$ and $G'$ with $\kappa= \frac{1}{2}|G \triangle G'|\leq n$ where $F$ does not contain a cycle of length $2\ell$. We want to show that the claim is correct for two $F$-fixed realisations $G$ and $G'$ with $\frac{1}{2}|G \triangle G'|= n+1$ where $F$ does not contain a cycle of length $2 \ell.$  If $G \triangle G'$ is not connected we can apply the induction hypothesis on each of the swap cycles (each corresponds to one component) of at most length $2n-2.$ This leads to a sequence $G_1,\dots,G_k$ (we merge for each component its sub-sequences) such that each $G_i \triangle G_{i+1}$ corresponds to a $j$-swap with $j \leq 2\ell-4.$ Each of theses sub-sequences $G'_{1},\dots,G'_{k'}$ for a component has at most length $k'\leq \frac{1}{2}|G'_1 \triangle G'_{k'}|.$ Hence, the whole sequence has at most length $k \leq \frac{1}{2}|G \triangle G'|.$ \\
We now assume that $G \triangle G'$ is one alternating cycle $C=(v_1,v_2,\dots,v_{2 n+2},v_1)$ of length $2 n+2$ starting with $\{v_1,v_2\} \in E(G).$ \\
\emph{If $C$ is not vertex disjoint} there must be a vertex, say $v_1,$ which occurs at least twice in $C.$ Hence, there exist adjacent vertices $v_2,v_{2 n+2},v_i,v_{i+2}$ of $v_1$ on $C$ which must be pairwise different because the corresponding edges belong to the symmetrical difference $G \triangle G'.$ It follows that $C$ decomposes in cycle swaps $C_1:=(v_1,\dots,v_i,v_1)$ and $C_2:=(v_1,v_{i+2},\dots,v_{2n+2},v_1)$ of lengths $4 \leq \ell_1,\ell_2 \leq 2n-2.$ Obviously, we have $\ell_1+\ell_2=2n+2.$ We apply cycle swap $C_1$ on $G$ and yield realisation $G^{*}$ which is $F$-fixed. Moreover, $|G \triangle G^{*}|= \ell_1$ and $|G' \triangle G^{*}|= \ell_2$ corresponds to $C_2.$ We apply the induction hypothesis on $G \triangle G^{*}$, and get a sequence $G_1,\dots,G_{k_1}$ with $G_1:=G$, $G_{k_1}:=G^{*}$  where $G_i \triangle G_{i+1}$ corresponds to a $j$-swap with $j \leq 2\ell-2$. Moreover, due to (ii) $k_1 \leq \frac{1}{2}\ell_1.$ We apply the induction hypothesis on $G^{*} \triangle G'$ and get a sequence $G'_1,\dots,G'_{k_2}$ with $G'_1:=G^{*}$, $G'_{k_2}:=G'$ and  $G_i \triangle G_{i+1}$ corresponds to $j$-swaps with $j \leq 2 \ell-2$. Furthermore, due to (ii) we get $k_2 \leq \frac{1}{2}\ell_2.$ We merge both sequences by setting $G'_{k_2}=G'_1$ and get a sequence which fulfills (i) for $G \triangle G'.$ Moreover, this sequence has length $k=k_1+k_2 -1\leq  \frac{1}{2}(\ell_1+\ell_2)-1=n.$ \\
\emph{If $C$ is vertex disjoint} we have to show that there exists a vertex pair $\{v_i,v_j\} \notin F$ on $C$ connecting an odd alternating path which has at most length three. Assume all such vertex pairs are edges in $F.$ Then we find with Proposition~\ref{prop:cycles_in_F} a cycle of length $2 \ell$ in $F$ in contradiction to our conditions. Without lost of generality we assume that $(v_1,v_j) \notin F$ where $j$ is an even label and connects the odd alternating path $P_1=(v_1,\dots,v_j)$ of length $2n-1 \geq \ell_1 \geq 3.$  We denote the remaining sub-path of length $2n-1 \geq \ell_2 \geq 3$ by $P_2=(v_j,\dots,v_{2 n +2},v_1).$ Obviously, we have $\ell_1+\ell_2=2n+2.$ Observe that edge $\{v_1,v_j\}\notin G \triangle G'$ otherwise $v_1$ has at least three incident edges in $C$ which contradicts the assumption that $C$ is vertex-disjoint. We get two cases.\\
(i) $(v_1,v_j) \in E(G) \cap E(G')$: We get the alternating cycle $C'=(v_j,P_2)$, and swap its edges to the $F$-fixed realisation $G^*$ with $|G^* \triangle G'|=\ell_2+1$ and $|G^* \triangle G|=\ell_1+1.$\\
(ii) $(v_1,v_j) \notin E(G)\cup E(G')$: We get the alternating cycle $C''=(P_1,v_1)$, and swap its edges to the $F$-fixed realisation $G^*$ with $|G^* \triangle G'|=\ell_1+1$ and $|G^* \triangle G|=\ell_2+1.$\\
In the first case we apply on $C'$ the induction hypothesis and get a sequence $G_1,\dots,G_{k_1}$ of fixed $F$-realisations with $G_1:=G$ and $G_{k_{1}}:=G^*$ where $k_1 \leq \frac{1}{2}(\ell_1+1).$ Furthermore, we apply the induction hypothesis on the alternating cycle $G^* \triangle G'$ and get a sequence $G'_1,\dots,G'_{k_2}$ of fixed $F$-realisations with $G'_1:=G^{*}$ and $G'_{k_2}=G'$ with $k_2 \leq \frac{1}{2}(\ell_2+1).$ Merging these two sequences and setting $G_{k_1}:=G'_1$ leads to the expected sequence of $F$-fixed realisations of length $k = k_1+k_2-1 \leq \frac{1}{2}(\ell_1+1+\ell_2+1)-1=n+1.$  In the second case we apply on $C''$ the induction hypothesis and get a sequence $G'_1,\dots,G'_{k_2}$ of fixed $F$-realisations with $G'_1:=G^*$ and $G'_{k_2}:=G'$ where $k_2 \leq \ell_1+1.$ Furthermore, we apply the induction hypothesis on the alternating cycle $G^* \triangle G$ and get a sequence $G_1,\dots,G_{k_1}$ of fixed $F$-realisations with $G_1:=G$ and $G'_{k_1}:=G^{*}$ with $k_1 \leq \frac{1}{2}(\ell_2+1).$ Merging these two sequences and setting $G_{k_1}=G'_1$ leads to the expected sequence of $F$-fixed realisations of length $k = k_1+k_2-1 \leq \frac{1}{2}(\ell_1+1+\ell_2+1)-1=n+1.$ 
\qed\end{proof}

For $\ell=3$ this result is not true. Recall from Theorem~\ref{theorem:connectedness_F_Fixed_realizations_swaps} that for this special case we need to forbid $3$-matchings in $F$ which is stronger than the exclusion of cycles of length six. Since the largest vertex-disjoint cycle in the fixed set $F$ of a bipartite graph $G=(V,U,E)$ with $|V|=n$ and $|U|=n'$ with $n<n'$ can have length $2n$, we have to swap alternating cycles of length $2n$ in $G$ in the worst case. Specifically, scenarios where set $F$ contains Hamiltonian cycles are not practicable because finding alternating cycles of length $2n$ in $G$ is $NP$-complete. In cases where the length of cycles in $F$ is not too long we yield results leading to quite implementable Markov chains for $F$-fixed realisations. Theorem~\ref{Theorem:connectedness_F_Fixed_realizations_cycles_length_2k} leads to the following result for $\ell=4.$

\begin{corollary}\label{Cor:connectedness_F_Fixed_realizations_cycles_length_eight}
Let $G=(V,U,E)$ and $G'=(V,U,E')$ be two different $F$-fixed realizations of a sequence $S$ and $F:=F_E \cup F_N$ such that $F$ doesn't contain a cycle of length eight. Then there exist realizations $G_1,\dots,G_k$ with $G_1:=G$, $G_k:=G'$ and $|G_i \triangle G_{i+1}|\leq 6$ where $G_i \triangle G_{i+1}$ corresponds to a $4$-swap, or, a $6$-swap, and $k\leq \frac{1}{2}|G \triangle G'|$.
\end{corollary}

Since a forest does not contain a cycle of length eight, we can state the following result. 

\begin{corollary}\label{Cor:connectedness_F_Fixed_Forest}
Let $G=(V,U,E)$ and $G'=(V,U,E')$ be two different $F$-fixed realizations of a sequence $S$ and $F:=F_E \cup F_N$ such that $F$ is a forest. Then there exist realizations $G_1,\dots,G_k$ with $G_1:=G$, $G_k:=G'$ and $|G_i \triangle G_{i+1}|\leq 6$ where $G_i \triangle G_{i+1}$ corresponds to a $4$-swap, or, a $6$-swap and $k\leq \frac{1}{2}|G \triangle G'|$.
\end{corollary}

An adapted Curveball algorithm needs to perform $6$-swaps and $4$-swaps. Verhelst~\cite{Verhelst2008} gives a combination of his Metropolis-Hastings algorithm which performs trades and single $6$-swaps for square matrices where the fixed elements of $F$ are $0$'s on the diagonal. We here introduce an extended Curveball algorithm which can perform many $6$-swaps in one step and works for all scenarios where $F$ does not contain a cycle of length $8.$ For this we define a \emph{circle trade} for three different adjacency lists $A_i,A_j,A_k$ in the following way; (i)  Determine the sets $A_{j-i}$, $A_{k-j}$, and $A_{i-k}$ with $A_{j-i}:=A_j \setminus (A_i  \cup F_i \cup F_j).$ Let us consider the smallest of the sets $A_{j-i}$, $A_{k-j}$, and $A_{i-k},$ say $A_{i-k}.$ (ii) We choose a subset $A'_i$ of $A_{i-k}$ uniformly at random, i.e. each possible subset has the same probability to be chosen. The number $x':=|A_i'|$ determines the size of the trade. (iii) We choose $x'$-subsets $A_j',$ $A_k'$ in  $A_{j-i}$, and $A_{k-j}$ uniformly at random. (iv) We add vertices $A'_j$ to $A_i$ and delete them in $A_j.$ Furthermore, we add vertices $A_k'$ to $A_j$ and delete them in $A_{k}.$ Lastly we add vertices $A_i'$ to $A_k$ and delete them in $A_i.$ 

Consider for example adjacency matrix $A:=\left(\begin{matrix} \bf{0}&1&0&1&0&0\\0&\bf{0}&1&0&1&0\\1&0&\bf{0}&0&0&1\end{matrix}\right)$ with $F_1=\{1\}$, $F_2=\{2\}$ and $F_3=\{3\}.$ Then we have $A_{1-3}=\{2,4\}$, $A_{2-1}=\{3,5\}$ and $A_{3-2}=\{1,6\}.$ A circle trade of size $x'=2$ replaces $A_1=\{2,4\}$ by $A'_2=\{3,5\},$ $A_2=\{3,5\}$ by $A'_3=\{1,6\}$, and $A_3=\{1,6\}$ by $A'_1=\{2,4\}.$  This yields matrix $B=\left(\begin{matrix} \bf{0}&0&1&0&1&0\\1&\bf{0}&0&0&0&1\\0&1&\bf{0}&1&0&0\end{matrix}\right)$. A circle trade of size $x'=1$ chooses for example subsets $A_1'=\{4\}$, $A_2'=\{3\}$ and $A_3'=\{6\}$, and yields $C=\left(\begin{matrix} \bf{0}&1&1&0&0&0\\0&\bf{0}&0&0&1&1\\1&0&\bf{0}&1&0&0\end{matrix}\right).$ Let us now consider another order of $i,j,k$ which changes the circle trade, i.e. $2,1,3$ instead of $1,2,3.$ It is clear that the element $3$ of list $A_2$ cannot be moved to list $A_3$ because this corresponds to the fixed non-edge $\{v_3,u_3\}$ in $F.$ We get $A_{2-3}=\{5\},$ $A_{1-2}=\{2\}$ and $A_{3-1}=\{6\}.$ We can apply one circle trade of size $x'=1$ in choosing $A_1'=\{4\}$, $A_2'=\{5\}$ and $A_3'=\{6\}$, and get $D:=\left(\begin{matrix} \bf{0}&1&0&0&0&1\\0&\bf{0}&1&1&0&0\\1&0&\bf{0}&0&1&0\end{matrix}\right).$ Observe that we get only two different constellations for circle trades for all $3!$ possibilities to choose an order of $i,j,k.$ Especially each constellation occurs with the same probability. As long as we draw each order with the same probability we can ignore it.

In step (i) of the description of a circle trade we choose a subset in $A_{i-k}$ uniformly at random, i.e. each subset has the same probability to be chosen. This can be done in assigning a boolean variable $n_i \in \{0,1\}$ to each element $j \in A_{i-k}.$ The string $n_1n_2 \dots n_{|A_{i-k}|}$ corresponds to a binary number which can be transformed in a natural number between $0$ and $2^{|A_{i-k}|-1}.$ We choose a random natural number in this range, transform it in its binary number, and yield for element $j$ a $0$ or a $1$. In the corresponding circle trade we put all $j$ with $n_j=1$ in set $A_i'.$ 

We propose the following \emph{Curveball algorithm with circle trades}.  We start with all adjacency lists of an $F$-fixed realisation $G$, and repeat the following steps $t$ times. (a) Choose with probability $\frac{1}{2}$ either an $F$-ignoring trade (see Section~\ref{sec:swaps}), or a circle trade. (b) For an $F$-ignoring trade we choose two adjacency lists uniformly at random, and apply a randomly chosen $F$-ignoring trade to yield $F$-fixed realisation $G'$ if possible. Otherwise we stay in the current $F$-fixed realisation $G.$ (c) For a circle trade we choose one after another three adjacency lists uniformly at random, and apply a randomly chosen circle trade if possible. This yields $F$-fixed realisation $G'$. Otherwise we stay in the current $F$-fixed realisation $G.$
With analogous arguments like in Section~\ref{sec:swaps} we extend the set of solvable problem classes to sets $H$ which can be partitioned in a set $F$ without an $8$-cycle, and a set $F^{*} \subset F'$ where $F'$ is the set of static and non-static edges. Using in a preprocessing step the proposed Gale-Ryser procedure of the last section above Theorem~\ref{theorem:connectedness_F_F'_Fixed_realizations_swaps}, we are able to partition a given set $H$ in polynomial time, and to search for possible cycles of length $8$ in $F.$ Edges or non-edges are static and cannot be contained in circle trades or $F$-ignoring trades. Hence, we can state the following result.

\begin{theorem}\label{theorem:connectedness_F_F'_Fixed_realizations_eight_cycles}
Let $G=(V,U,E)$ and $G'=(V,U,E')$ be two different $H$-fixed realizations of a sequence $S$ where $H$ can be partitioned in sets $F$ and $F^{*}$ such that 
\begin{enumerate}
\item $F$ does not contain cycle of length $8$, 
\item $F^{*} \subset F'$ where $F'$ is the static edge and non-edge set of $S$, and 
\item $F \cap F^{*} = \emptyset$.
\end{enumerate}
Then the $H$-ignoring Curveball algorithms with circle trades samples an $H$-fixed realisation uniformly at random for $t \mapsto \infty.$
\end{theorem}

\begin{proof}
Let us start with the $F$-ignoring Curveball algorithms with circle trades. With the fundamental theorem for Markov chains the statement is true if the chain is irreducible, reversible and aperiodic. Since there is always a sequence of $4$-swaps and $6$-swaps between two realisations with Theorem~\ref{Theorem:connectedness_F_Fixed_realizations_cycles_length_2k}, the chain is irreducible because we can define a sequence of $F$-ignoring trades (which are $4$-swaps), and circle trades ($6$-swaps) between each pair of realisations. A circle trade or an $F$-ignoring trade between two realisations $G$ and $G'$ can always been done in both directions which makes the state digraph symmetrical. We need to prove that we find transition probabilities $P(A,B)=P(B,A)$ for all matrices $A$ and $B$ which differ by one $F$-ignoring trade or a circle trade. In this case the Markov chain is reversible, i.e. $\pi(A)P(A,B)=P(B,A)\pi(B)$ and the stationary distribution $\pi$ must be uniform. For an $F$-ignoring trade this condition was already proven in Corollary~\ref{Cor:Curveball_F_fixed} with the little difference that we here have $P(A,B)=\frac{1}{n \cdot (n-1)}{|A_{i-j} \cup A_{j-i}|\choose |A_{i-j}|}^{-1}$. It does not change anything on the proof idea from the mentioned corollary.  Hence, here we can focus on circle trades. Let us consider sets $A_{i-k}$ and $B_{i-k}$ for two matrices $A$ and $B$ which differ by one circle trade for adjacency lists $A_i$, $A_j$ and $A_k$. In a circle trade we always choose $x'$-subsets with $1 \leq x' \leq \min\{|A_{i-k}|,|A_{j-i}|,|A_{k-j}|\}.$ Let us denote this minimum by $m.$ Then there are $2^m$ possible circle trades for one chosen constellation of adjacency lists. (Recall that we have $3!$ possible orders of adjacency lists but only two different constellations for circle trades.) Hence, the transition probability from $A$ to $B$ is $P(A,B)=\frac{3}{n \cdot (n-1) \cdot (n-2)}\frac{1}{2^m}.$ The factor $3$ in this term occurs since three different orders of chosen $i,j,k$ lead to the same circle trade. The circle trade in $B$ which is necessary to come from $B$ to $A$ uses the whole trade in the opposite direction, i.e. it works on sets $B_{k-i},$ $B_{i-j}$ and $B_{j-k}$. We need to show that the minimum of these sets is also of size $m.$ Each circle trade from $A_j$ to $A_i$, $A_k$ to $A_j,$ and $A_i$ to $A_k$ can be done in the opposite direction from $B_i$ to $B_j$, $B_j$ to $B_k,$ and $B_k$ to $B_i.$ Hence, $m$ must be the same between $A$ and $B$ which differ by a circle trade. Hence, we find $P(A,B)=P(B,A).$\\
The chain is aperiodic since we always find a $4$-swap or a $6$-swap in a matrix $A$ whenever there are at least two $F$-fixed realisations with Theorem~\ref{Theorem:connectedness_F_Fixed_realizations_cycles_length_2k}. These swaps either corresponds to (a) an $F$-ignoring trade between two lists $A_i$ and $A_j,$ or, (b) a circle trade between three lists $A_i,$ $A_j$ and $A_k.$ As soon as case (a) occurs in one single realisation $A$ we also have the trade where $B_{i,j}=A_{i-j}$, and stay in the current realisation. If we only have case (b) for all realisations there must be fixed edges and non-edges in $F$ in each matrix $A.$  Otherwise we find case (a). Then we have two different constellations for three chosen adjacency lists for a circle trade, i.e. either $i,j,k$ or $j,i,k.$ Assume a circle trade for both orderings is always possible. Then it is easy to construct a $4$-swap for two of these lists since all couples $A_{i-j},A_{j-i}$ are not empty. Hence, there is one constellation which cannot lead to a suitable circle trade showing that the $F$-ignoring chain with circle trades is aperiodic.\\
It is clear that each $F$-fixed realisation is also an $H$-fixed realisation and vice versa. Especially edges in $F^*$ cannot be contained in suitable $F$-ignoring trades or circle trades. Hence, each $F$-ignoring trade can be modified to an $H$-ignoring trade, and in circle trades we can exclude edges from $H.$ Hence, the $H$-ignoring Curveball algorithm with circle trades is an ergodic Markov chain too.
\qed\end{proof}
 
Let us consider a special sub-case of this set $F$, i.e. $F$ only consists of non-edges which form a perfect matching. Then $F$ does not contain a cycle of length $8.$ In this case, we can relabel some vertices such that all $0$'s in $F$ occur in the diagonal of a suitable $F$-fixed realisation matrix $A$ of a bipartite sequence $S=(a_1,\dots,a_n),(b_1,\dots,b_{n'})$. Now interpreting $A$ as the adjacency matrix of a directed degree realisation without loops of the directed sequence $S_d=((a_1,b_1),\dots,(a_n,b_n))$ leads to the well-known analogous result of Rao et al.~\cite{Rao1996}, that there exists an ergodic sampling chain which is based on directed $4$-swaps, and the reorientation of directed cycles of length three. Directed cycles of length three correspond in this bipartite interpretation to $6$-swaps. However, Berger et al. showed in \cite{Berger2010} that it is sufficient to reorient a special kind of directed cycles to get an ergodic sampling chain. These are cycles which are contained in one of the possible directions in each directed realisation of $S_d.$ We call the vertices of one of those cycles \emph{induced cycle set}, and $C_F$ the set of all induced cycles for $S_d$. Induced cycle sets were also shown to be pairwise disjoint. Moreover, they also showed that it is possible to avoid the reorientation of induced cycles in a sampling chain, and only to use swaps. The reason is that in this case the state graph decomposes in $2^k$ \emph{isomorphic components} where $k:=|C_F|\leq \frac{n}{3}.$ This means it is possible to choose one of these components uniformly at random, and then to start a random walk in this component, i.e. only swaps need to be applied. Unfortunately, the nice property of isomorphic components  breaks in our more general case if we avoid circle trades. The state graph in Figure~\ref{fig:counter_example} decomposes in two non-isomorphic components if $6$-swaps are avoided. The set $F$ is here a matching of size $3$ and only one fixed additional edge. 

\begin{figure}
 \centering
 \includegraphics[scale=0.2]{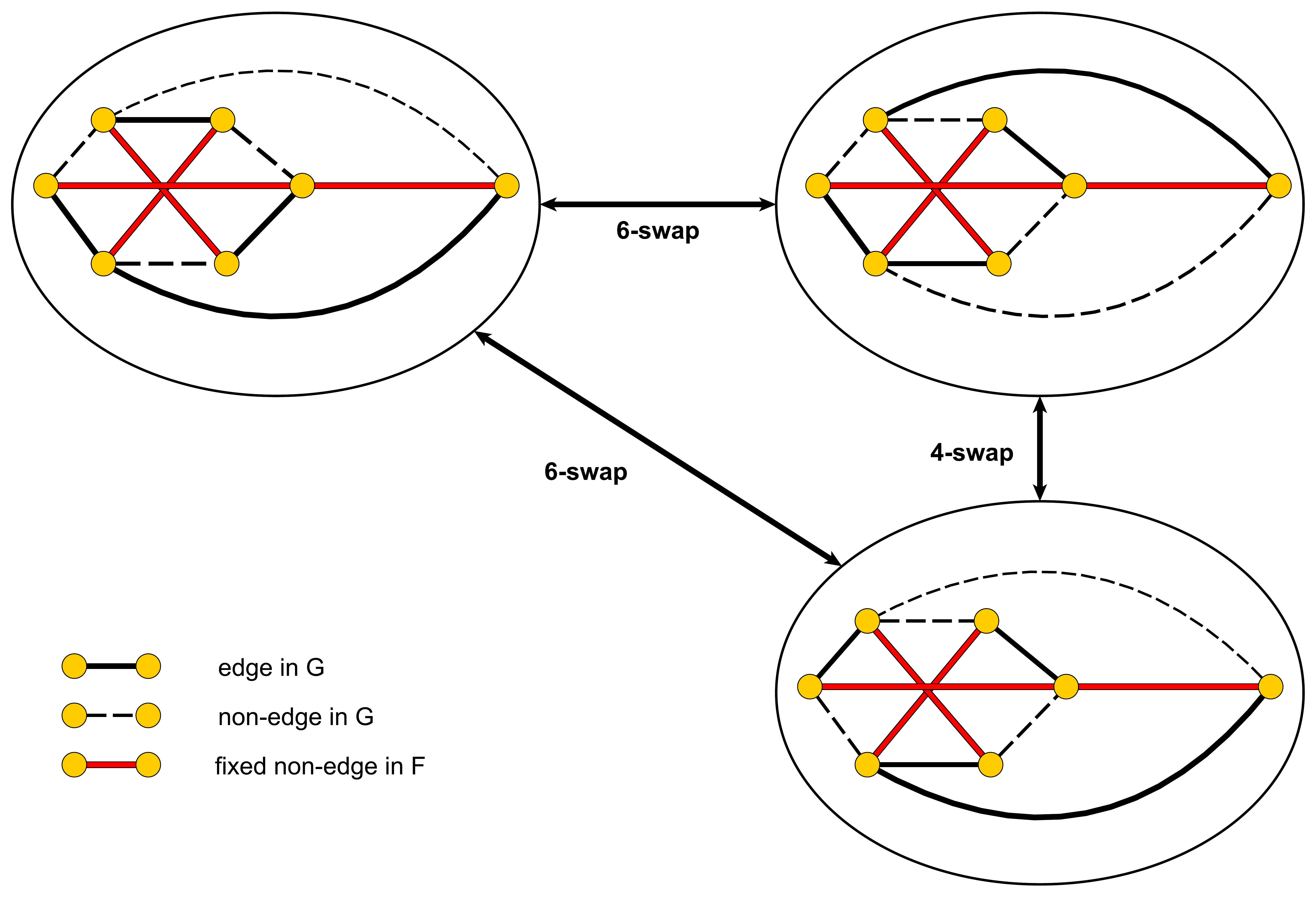}
\caption{State graph for $F$-fixed problem with sequence $S=(1,1,1,1),(2,1,1)$, all elements in $F$ are non-edges.}
\label{fig:counter_example}
\end{figure}

\section{Instructional Summary}

We want to summarize our results in giving an instruction for sampling $H$-fixed realisations for a sequence $S.$ Given a bipartite graph $G=(U,V,E)$ with fixed edges $F_E$ and non-edges $F_N,$ i.e. $F_E \subset E,$ and $F_N \cap E = \emptyset.$ We apply the following steps to sample an $H$-fixed realisation.
\begin{enumerate}
\item Determine for vertex degree sequence $S=(a_1,\dots,a_n),(b_1,\dots,b_n)$ of $G$ set $F'$ of all edges and non-edges which occur in \textbf{every} realisation. This can be done in polynomial time with the repeated use of the Gale-Ryser theorem, see the description after Corollary~\ref{Cor:Curveball_F_fixed}.
\item Update the fixed sets of edges and non-edges in $G$ to set $F:=H \setminus F'.$
\item If set $F$ does not contain a matching of size three then it is sufficient to use $4$-swaps in a Markov chain. An $F$-ignoring swap chain or  Curveball algorithm can be used to sample (after Theorem~\ref{theorem:connectedness_F_Fixed_realizations_swaps}).
\item Else, if set $F$ does not contain a cycle of length $8$, then it is sufficient to use $4$-swaps and $6$-swaps in a suitable Markov chain. A Curveball algorithm with circle trades which is defined above Theorem~\ref{theorem:connectedness_F_F'_Fixed_realizations_eight_cycles} can be used to sample.
\item Else, if in all other cases one can try to determine if $F$ does not contain a cycle of length $10,12,\dots,2k$ such that $k$ is not too large. In this case swapping of alternating cycles with shorter length than $2k$ delivers a sampling algorithm. This algorithm can be applied depending on the size of $k$ and the size of realisations. 
\end{enumerate}

We propose in future endeavours to do a lot of experimental work. Whereas we cannot answer how efficient these methods work for sampling $H$-fixed realisations in theory, we should find out in experiments if several metrics, which were proposed for stopping rules see~\cite{Aldous86} converge slower or faster than for the corresponding realisation problem without the fixed set $F$ of edges and non-edges. It will also be interesting to observe how sets $H$ look like in real world applications to see 'how far we have come' in practice with these new algorithms.

\end{document}